\documentclass[12pt]{amsart}
\usepackage{amssymb,amsthm,amsmath}
\usepackage{graphics,psfrag,graphicx,subfigure,epsfig}
\usepackage{enumerate}
\usepackage[all]{xy}
\usepackage{mathtools}
\usepackage{euscript}
\usepackage{eufrak}
\usepackage{mathrsfs}
\def\f{\mathcal{F}}
\def\fo{\mathfrak{F}}
\def\g{\mathcal{G}}
\def\fin{\mathcal{FIN}}
\def\vcy{\mathcal{VCY}}
\def\all{\mathcal{ALL}}
\def\evyc{\underline{\underline{E}}G}
\def\efin{\underline{E}G}

\def\nh{N_{G}[H]}
\def\gh{\mathcal{G}[H]}
\def\z{\mathbb{Z}}
\def\zz{\mathbb{Z}\rtimes \mathbb{Z}}

\def\r{\mathbb{R}}
\def\l{\ell}
\def\L{\mathcal{L}}
\def\a{\mathcal{A}}
\def\q{\mathbb{Q}}

\def\nu{N(\mathcal{U})}

\def\Iso{\rm{Iso}}
\def\s{\mathbb{S}}
\newtheorem{thm}{Theorem}[section]
\newtheorem{prop}[thm]{Proposition}

\theoremstyle{definition}

\newtheorem{defi}[thm]{Definition}

\newtheorem{rem}[thm]{Remark}

\newtheorem{lem}[thm]{Lemma}

\title{Models for classifying spaces for $\zz$}

\author{Daniel Juan-Pineda}
\address{Centro  de  Ciencias Matem\'aticas. \\
UNAM Campus  Morelia\\
Ap.Postal  61-3 Xangari\\ Morelia, Michoac\'an.  M\'EXICO 58089}
 \email{daniel@matmor.unam.mx}
\thanks{We acknolwedge support from research grants from DGAPA-UNAM and CONACyT-M\'exico}

\author{Alejandra Trujillo-Negrete}
\address{Centro  de  Ciencias Matem\'aticas. \\
UNAM Campus  Morelia\\
Ap.Postal  61-3 Xangari\\ Morelia, Michoac\'an.  M\'EXICO 58089}
 \email{aletn@matmor.unam.mx}

\begin{document}

\begin{abstract}
  We construct two models for the classifying space for the family of infinite cyclic subgroups of the fundamental group of the Klein bottle.
  These examples do not fit in general constructions previously done, for example, for hyperbolic groups.

\end{abstract}
\maketitle
\section{Introduction}

Let $G$ be a discrete group. A \emph{family}, $\f$, of subgroups of $G$ is a nonempty set of subgroups of $G$
which is closed under conjugation and taking subgroups.
A model, $E_{\f}(G)$, for the \textit{classifying space  of the fami\-ly $\f$} is
a $G$-CW-complex $X$, such that all of its isotropy groups  belong to $\f$ and if
$Y$ is a  $G$-CW-complex  with isotropy groups belonging to $\f$, there is precisely one
$G$-map $Y \rightarrow X$ up to $G$-homotopy. Classifying spaces for families appear frequently in mathematics, notably, in various Assembly Isomorphism conjectures
such as  the Baum--Connes  and the Farrell--Jones Conjectures, see \cite{luck}.

A model for $E_{\vcy } (\z\times \z)$, the classifying space for the family of virtually
cyclic subgroups of $\z \times \z$, may be constructed as countable join $*_{i \in \z}\r_i$
with each $\r_i \simeq \r$ see \cite{JL}, with a suitable action of $\z\times\z$.
This space is built by using the fact that each nontrivial virtually cyclic subgroup $H$
of $\z \times \z$ is normal in $\z\times \z$. We cannot build a model for $E_{\vcy}(\zz) $
in the same way because  this is not the case in $\zz$. In this note, we present two models
for $E_{\vcy}(\zz )$ (which are $\zz$-homotopy equivalent). Our principal results are Theorems \ref{thm-join} and \ref{second-model}, and Proposition \ref{prop-homology}.

The constructions follow the 
work of W. L\"uck and M. Weiermann in \cite{luck-weiermann}, and of D. Farley
in \cite{farley}. The latter uses the fact that $\zz$ is a $CAT(0)$ group as it acts
 by isometries on the plane, and the former follows a general construction. 
 
We thank the referee for valuable suggestions. 


\section{Classifying Spaces for Families}

Let $G$ be a discrete group. A \emph{family} $\f$ of subgroups of $G$ is a nonempty set of subgroups of $G$
which is closed under conjugation and taking subgroups. Some examples  are:
$ \{1\}$, the family consisting of the trivial subgroup in  $G$,
  $\fin$, the family of finite subgroups of $G$,
  $\vcy$ the family of virtually cyclic subgroups of $G$ and
  $\all$, the family of all subgroups of $G$.

Let $H$ be a subgroup of $G$, and $\f$ be a family of subgroups of $G$,  then $\f$ defines a family of $H$ as follows
\begin{equation*}
\f (H)=\{K\subseteq H| K\in \f  \}. \label{fsubg}
\end{equation*}

\begin{defi}Let $\f$ be a family of subgroups of $G$. A model for the
\textit{classifying space $E_{\f}(G)$ of the family $\f$} is a $G$-CW-complex $X$,
such that all of whose isotropy groups  belong to $\f$. If $Y$ is a  $G$-CW-complex
with isotropy groups belonging to $\f$, there is precisely one map $G$-map $Y \rightarrow X$ up
to $G$-homotopy. We denote by $B_{\f}(G)$ be the quotient of $X$ by the action of $G$.
\end{defi}
In other words, $X$ is a terminal object in the  category of $G$-CW complexes with isotropy groups belonging to $\f$.
In particular, two models for $E_{\f}(G)$ are $G$-homotopy equivalent, and then we denote $X$
by  $E_{\f}(G)$.

\begin{rem}\label{incfam}
Given two families $\f_1\subseteq \f_2$ of subgroups of $G$,  since $E_{\f_2} (G)$ is a terminal object in the  category of $G$-CW complexes with isotropy groups belonging to $\f_2$,     there exists precisely one $G$-map up to $G$-homotopy
$$E_{\f_1}(G) \rightarrow E_{\f_2}(G).$$
\end{rem}

\begin{defi}
Let $G$ be a group, $H\subseteq G$ and $X$  a  $G$-set.  The \emph{fixed point set}  $X^H$ is  defined as
\begin{align}
X^H =\{x\in X \mid \forall h\in H, \; hx =x  \; \}.
\end{align}
\end{defi}

\begin{thm}\cite[Thm. 1.9]{luck} \label{luck-caract}
A $G$-CW-complex $X$ is a model for $E_{\f}(G)$ if and only if the $H$-fixed point set $X^H$ is contractible for $H \in \f$ and is empty for $H \not\in \f$.
\end{thm}
A model for $E_{\mathcal{ALL}}(G)$ is $G/G$, a model for $E_{\{1\}}(G)$ is the same as a  model for $EG$, the total space of the universal $G$-principal bundle $EG \rightarrow BG$, \cite{milnor}.

We write  $\efin$ for $E_{\mathcal{FIN}(G)}$, this is known as the \textit{universal} $G$-CW-\textit{complex for proper $G$-actions}, and we write $\evyc$ for $E_{\mathcal{VCY}}(G)$.


\subsection{Constructing models from models for smaller families} \label{section-build}

Let $\f$ and $\g$ be two families of a group $G$, with $\f \subseteq \g$, such that we
know a model for $E_{\f}(G)$. In \cite{luck-weiermann} W. L\"uck and M. Weiermann build a model
$E_{\g}(G)$ from $E_{\f}(G)$, and in \cite{farley} Farley builds another model for groups
acting on CAT(0)-spaces. In this Section we present these models.

Let $\f=\fin$ and $\g=\vcy$. Following \cite{luck-weiermann}, define an equivalence relation $\sim$ on $\fo:= \vcy - \fin$  as
   \begin{equation} \label{rel-eq}
   V\sim W \:\:\Longleftrightarrow  |V \cap W |= \infty ,
   \end{equation}
for $V$ and $W$ in $\fo$, where
$|\star|$ denotes the cardinality of the set $\star$.
Let $[\fo]$ denote the set of equivalence classes under the above relation
and  let $[H] \in [\fo ]$  be the equivalence class of $H \in \fo$.

For $H\in \f$
\begin{align}\label{nh}
 \nh :=&\{ g\in G \mid [g^{-1} Hg]=[ H]  \}     \nonumber
\\ =&\{ g\in G \mid |g^{-1} H g \cap H|= \infty \}.
 \end{align}
This is  the isotropy group of $[H]$ under the $G$-action on $[\fo]$ induced by conjugation.
Note that  $\nh$ is the commensurator of $H$ in $G$. 
Define a family of subgroups $\gh$ of $\nh$ by
\begin{align} \label{gh}
  \gh:=\{K \subseteq \nh \mid K\in \fo, |K\cap H|= \infty \} \cup \{ \fin \cap \nh \}.
\end{align}

The method to build a model of $E_{\g}(G)$ from one of $E_{\f}(G)$ is with the following theorem.
\begin{thm}\cite[Thm. 2.3]{luck-weiermann}\label{thm-luck-weierm}
Let $\f \subseteq \g$ and $\sim$ as above. Let $I$ be a complete system of representatives $[H]$ of the $G$-orbits in $[\g-\f]$ under the $G$-action coming from conjugation. Choose arbitrary $N_G[H]$-CW-models for $E_{\f\cap N_G[H]}(N_G[H])$,  $E_{\g[H]}(N_G[H])$ and an arbitrary $G$-CW-model for $E_\f(G)$. Define $X$ a $G$-CW-complex by the cellular $G$-pushout
\begin{align*}
\xymatrix{
                \coprod_{[H]\in I} G \times_{N_G[H]} E_{\f\cap N_G[H]}(N_G[H])
                \ar[d]^{\coprod_{[H]\in I} id_G \times_{N_G[H]}f_{[H]}}
                \ar[r]^(0.75){i}
                &E_{\f}(G)
                \ar[d]
                \\
                \coprod_{[H]\in I} G \times_{N_G[H]} E_{\g[H]}(N_G[H])
                \ar[r]
                & X
}
\end{align*}
such that $f_{[H]}$ is a cellular $N_G[H]$-map for every $[H]\in I$ and $i$ is an inclusion of $G$-CW-complexes, or such that every map $f_{[H]}$ is an inclusion of $N_G[H]$-CW-complexes for every $[H]\in I$  and $i$ is a cellular $G$-map. 
Then $X$ is a model for $E_{\g}(G).$
\end{thm}
The maps in Theorem \ref{thm-luck-weierm} are given by the universal property of  classifying spaces for families and  inclusions of families of subgroups (see Remark (\ref{incfam})). \\

The following  is Definition 2.2, in \cite{farley}.
\begin{defi}\label{def-i-g-g}
Let $\f \subseteq \g $ be families of subgroups of a group $G$. We say that a $G$-CW complex $X$ is an \emph{$I_{\g -\f} G$-complex}  if
\begin{enumerate}
\item[(i)] whenever $H \in \g-\f $, $X^H$ is contractible;
\item[(ii)] whenever $H\not\in \g $, $X^H=\emptyset $.
\end{enumerate}
\end{defi}

Observe that if all isotropy groups are in $\g$ then (ii) holds. Since the trivial subgroup
is not in $\g- \f$,  $X$ is not necessarily contractible.

\begin{thm} \cite[Prop. 2.4]{farley} \label{join-g-f}
 If $G$ is a group and $\f \subseteq \g$ are families of subgroups of $G$, then the join
$$(E_{\f})*(I_{\g -\f}G)$$
is a model for the classifying space  $E_{\g}G$.
\end{thm}


\section{First  model for $\underline{\underline{E}}({\z\rtimes \z})$} \label{section-modelcat}

Let $\mathcal{K}:=(\zz )\backslash \r^2$ be the klein bottle. Following \cite{farley} we construct a model for $\underline{\underline{E}}({\z\rtimes \z})$ using the fact that $\zz $ acts  by isometries  on  $\r^2$. This action is given by deck transformations of the universal covering $p\colon \r^2 \to \mathcal{K}$ of the Klein bottle as $\zz\cong \pi_1(\mathcal{K})$.

\subsection{Virtually cyclic subgroups of $\zz$}

 Let  $\zz$  be the \emph{Klein bottle  group} with multiplication
\begin{align*}
  (n_1, m_1)(n_2, m_2)=(n_1+(-1)^{m_1}n_2, m_1+m_2),
\end{align*}
 inverse element
 $$(n,m)^{-1}=((-1)^{1-m}n,-m)$$
  and the neutral element is  $(0,0)$.

\begin{rem} \label{remconj}
 For $(t_1,t_2), (n,m) \in \zz$ we have that
  \begin{align}
  (t_1,t_2)(n,m)(t_1,t_2)^{-1} &=(t_1,t_2)(n,m)((-1)^{1-t_2} t_1, -t_2)    \nonumber  \\
  &= (t_1,t_2) (n+(-1)^m (-1)^{1-t_2} t_1, \; m-t_2   ) \nonumber  \\
  &= (t_1 + (-1)^{t_2} [n + (-1)^{(m+1-t_2)} t_1] ,  \; m   )  \nonumber \\
  &=( (-1)^{t_2}n + t_1 + (-1)^{m+1}  t_1, \; m) \label{conj}
  \end{align}
\end{rem}
In $\zz$ the families $\fin$ and $\vcy$ are
\begin{align*}
  \fin &=\{1\}\\
\vcy &=\{C\subseteq \zz \mid \, C\;\text{is infinite cyclic}\}\cup\{1 \}.
\end{align*}
\begin{rem}\label{rem-clasif-subg}
We  classify all infinite cyclic subgroups of  the family $\vcy$ in $\zz$. Let $n,m\in \z$, $k \in \mathbb{N}$. Observe that
    \begin{align*}
   &(n,m)^k= ([1+(-1)^m +(-1)^{2m}+ \cdots (-1)^{(k-1)m }] n, \: km) ;\\
   &(n,m)^{-k}=([(-1)^{1-m} + (-1)^{1-2m}+ \cdots + (-1)^{1-km}  ] n, \; -km).
    \end{align*}

Therefore infinite cyclic subgroups in $\vcy$  are of the following form
     \begin{align}
     & \langle(n,2m')\rangle =\{(kn,2km') \mid k\in \z\}, \;\;   \text{where} \;\; (n,2m')\neq (0,0) ;  \label{power-e}  \\
      & \langle(n,2m'+1)\rangle = \{ (n,2m'+1)^k \mid  k\in \z \}, \;\;\; \text{where}      \label{power-o-e} \\
             &\;\;\;\;\;  (n,2m'+1)^k =\left \{ \begin{array}{cc} (0,k(2m'+1)), & \;\; \text{if $k$ is  even }
                                                            \\ (n,k(2m'+1)) & \;\; \text{if $k$ is odd}.
                                                           \end{array} \right.  \nonumber
          \end{align}
\end{rem}

\subsection{A model for $\underline{\underline{E}}({\z\rtimes \z})$}

The group $\zz$ acts on $\r^2$  by deck transformations of the universal covering
$p\colon \r^2 \to \mathcal{K}$ of the Klein bottle. Explicitly, the action is as follows: let $(n,m)\in \zz$ and  $(t,r)\in \r^2$,  then
   \begin{align}\label{action-r2}
   (n,m)(t,r)=(n+(-1)^m t, m+r).
      \end{align}
  \begin{rem} \label{r2-free}
  A model for $E(\zz)$ is $\r^2$, because  $\r^2$ is contractible and the action (\ref{action-r2}) is free and properly discontinuous.
  \end{rem}

Let $\l(a,b)$ denote   the geodesic line in $\r^2$ determined by $a,b \in \r$ as
   \begin{align}\l(a,b):=\{(x, ax +b) \mid x \in \r \}, 
   \end{align}
 and  let
   \begin{align}\l(\infty, b):=\{ (b, y) \mid y \in \r \},     \end{align}
 denote the geodesic line parallel to $y$-axis, determined   by $b$.   Let  $\L$ be the \emph{ space of lines}  in $\r^2$:
    \begin{align}
    \L:=\{\l(a,b) \mid a \in \r\cup\{\infty \}, \;\; b \in \r \}.
    \end{align}
The space of lines is a metric space with the following distance
\begin{align*}
  d(\l_1,\l_2)= \left \{
                \begin{array}{cc}
                \frac{k}{1+k}  &  \text{if $\l_1$ and $\l_2$ bound a flat strip of width }k; \\
                1 & \text{if $\l_1$ and $\l_2$ are not parallel.}
                \end{array}
                \right.
\end{align*}
Then  we have that    $\L= \coprod_{a\in \r \cup \{\infty \}}   \r_a  $, where
$\r_a:=\{\l(a,b)\mid b\in \r\},$ and the metric in each connected component $\r_a$ is given  by $d$.

Since the action of $\zz$ in $\r^2$ sends geodesic lines to geodesic lines, it induces an action of $\zz$ on $\L$.  For $(n,m)\in \zz$ and $\l(a,b)\in \L$,  this action  is as follows,
   \begin{align}
   &(n,m)\l(a,b)=\l((-1)^m a, b+m-(-1)^m an ), \;   \text{ if } a \in \r  ; \label{action-lines} \\
   &(n,m)\l(\infty,b)= \l(\infty, n+ (-1)^m b ).  \label{action-lines2}
   \end{align}

\begin{defi}
 Let $(n,m)\neq (0,0)$ in $\zz$. We say that a line $\l \subset \r^2$ is an \emph{axis} for
 $(n,m)$   if
   $(n,m)\l =\l $ and  $(n,m) $  acts by translation on $\l $.
 \emph{The axis space}, for elements of $\zz$ in $\r^2$, is defined as follows
     \begin{align*}
        \a& :=\{\l \in \L  \mid \l \;\: \text{is an axis for some }\; (n,m)\in \zz-(0,0) \} .
       \end{align*}
\end{defi}
By (\ref{action-lines2}) all the lines $\l(\infty,b)$ are axes. And by  (\ref{action-lines})   we have
 \begin{align}
(n,m) \l(a,b) =\l(a,b)  \; \;\;& \text{iff } \;\;   \l((-1)^m a, b+m-(-1)^m an )= \l(a,b) ;       
\nonumber \\
   & \text{iff }  \; \;  m \; \text{even   and } \;  a=\tfrac{m}{n} . \label{arterisco}
 \end{align}
Observe that if $(n,m)$ fixes  $\l(a,b)$, then it acts on $\l(a,b)$  by translation.
Therefore  $\l(a,b) \in \a $ if only if $a \in \q\cup \infty $, and then
\begin{align}
        \a   &= \{ \l(a,b) \in \L \mid \;   a\in \q \cup \infty , \; b \in \r  \}        \nonumber \\
         &= \coprod_{a\in \q \cup \{\infty \}}   \r_a.   \label{axis-coprod}
     \end{align}


\begin{prop}\label{axis}
The axis space $\a$ of $\r^2$ is an $I_{\vcy-\{1\}}(\zz)$-complex.
\end{prop}
The proof of Proposition \ref{axis} is given  in the next Subsection.
   \begin{thm}\label{thm-join}
  A model for the  classifying space for the family of virtually cyclic subgroups of $\zz$ is  the join
  \begin{align} \label{e-vcy-klein}
  \underline{\underline{E}}(\zz) = \r^2 * \coprod_{a\in \q \cup \{\infty \}} \r.
     \end{align}
 Therefore, the quotient by the action is
  \begin{align}\label{b-vcy-klein}
  \underline{\underline{B}}(\zz) = \mathcal{K} * \coprod_{a \in \q_{\geq 0} \cup \{\infty\} }\s^1.
  \end{align}
\end{thm}
\begin{proof}
By Proposition \ref{axis}, the axis space of $\r^2$, is an  $I_{\vcy-\{1\}}(\zz)$-complex.
    Thus, by Theorem \ref{join-g-f} and because
 $\r^2$ is a model for $E(\zz)$, we have (\ref{e-vcy-klein}).  On the other hand,
 (\ref{b-vcy-klein}) easily follows by looking the  action of $\zz$ on $\mathcal{K}$ and $\a$,
 since the action of $\zz$ on $\a$ is by translation.
\end{proof}


\subsection{Proof of Proposition \ref{axis}}

 In this Section we will prove that the axis space $\a$ is an
 $I_{\vcy-\{1\}}(\zz)$-complex, which follow from Lemmas \ref{isotropy} and \ref{fixed}, below.

 We denote by $\Iso(a,b)$ the isotropy subgroup of $\l(a,b)\in \a$, where  $a\in \q \cup \{\infty\}$, and  $b\in \r$, that is,
      \begin{align*}
         \Iso(a,b):=\{ (n,m)\in \zz \mid  (n,m)\l(a,b)=\l(a,b)\}. 
    \end{align*}       
\begin{lem} \label{isotropy}
All isotropy subgroups  of the action of $\zz$ in the axis space $\a$ are in the set $\vcy-\{1\}$.
\end{lem}
\begin{proof}
 We  compute the isotropy subgroups of the action of $\zz$ on $\a$. See Remark \ref{rem-clasif-subg} about the family $\vcy$ of $\zz$.
\begin{itemize}
\item[(i)] If $a \in \q- \{0\}$ and $b \in \r$, by (\ref{arterisco}), we have
    \begin{align*}
       (n,m) \in \Iso(\l(a,b))                                      &  \; \text{iff}  \;  m \;\text{is even  and } \; a= \tfrac{m}{n}.
      \end{align*}
Suppose $a=\tfrac{a_1}{a_2}$ with $gcd(a_1,a_2)=1$,  then:
  \begin{enumerate}
  \item[1.] If $a_1$ is even then
     $$ \Iso(a,b)=\langle (a_2,a_1) \rangle \in \vcy-\{1\}.$$
  \item[2.] If $a_1$ is odd then
     $$\Iso(a,b)=\langle (2a_2,2a_1) \rangle \in  \vcy-\{1\}.$$
  \end{enumerate}
\item[(ii)] Now if $a=0$, by (\ref{action-lines}) we have
          $$ \Iso(0,b)=\langle (1,0)  \rangle \in \vcy-\{1\}.$$
\item[(iii)] Lastly,  if $a=\infty$, by (\ref{action-lines2}) we have
\begin{align*}
 (n,m)\in \zz \in \Iso(\infty,b) \; & \text{ iff } 
    n=(1-(-1)^m) b.
\end{align*}
If $m$ is even, then $n=0$, and if $m $ is odd, $n=2b $, but $b \in \r$ and $n \in \z$,  then we conclude that
\begin{enumerate}
\item[1)] if $2b \in \z$  then by (\ref{power-o-e})
\begin{align*}
\Iso(\infty, b)&=\langle (0,2) \rangle \cup \{(2b,2m'+1) \mid m' \in \z\} \\
                     &=\langle (2b,1) \rangle \in  \vcy-\{1\};
\end{align*}
\item[2)] if $b\in \r$ and $2b \not\in \z$  then by (\ref{power-e}) 
\begin{align*}
\Iso(\infty,b) = \langle (0,2) \rangle \in \vcy-\{1\}.
\end{align*}
\end{enumerate}
\end{itemize}
\end{proof}

\begin{lem}\label{fixed}
If $H\in \vcy - \{1\}$, the fixed point set $\a^H$ is contractible.
\end{lem}
\begin{proof}
Let $H=\langle (n,m)\rangle \in \vcy-\{1\}$,  a infinite cyclic subgroup of $\zz$  (see (\ref{power-e}) and (\ref{power-o-e})),   we will describe
    $$\a^H= \{\l \in \a \mid  H \l =\l   \},$$
along the following lines.
\begin{itemize}
\item[(i)] Suppose $m$ is even, and $n \neq 0$,  then by (\ref{power-e}),  $$H=\{(kn,km) \mid k \in \z \}.$$
Observe by (\ref{action-lines2}) that 
$$(kn,km)\l(\infty,b)=\l(\infty, kn+b).$$  Since $n \neq 0$, then  $\l(\infty, b) \not \in \a^H$ for every $b \in \r$. \\
 Now if $a \neq \infty$ and $b\in \r$, by (\ref{arterisco}),  we have
    \begin{align*}
    \l(a,b) \in \a^H \;\; & \text{iff }  a=\frac{km}{kn}=\frac{m}{n}.  
     \end{align*}
    Therefore
          $$ \a ^H = \r_{\frac{m}{n}},$$
  which is contractible.
\item[(ii)] Let $m\neq 0$ even  and $K=\langle (0,m)\rangle =\{(0,km ) \mid k\in \z \}$.  By (\ref{action-lines})
$$(0,km )\l(a,b)=  \l(a, b + 2km ),$$
we have that $\l(a,b) \not \in \a^K$  whenever  $a \in \q$.  \\
On the other side, by (\ref{action-lines2}),
$$(0,km)\l(\infty,b)= \l(\infty, b), \;\:\forall k \in \z,\text{ and }\forall b \in \r .$$
Therefore,
    $$\a^K= \r_{\infty},$$
which is contractible.
\item[(iii)] If $m$  is odd and $R=\langle (n,m)\rangle  $,  recall  from (\ref{power-o-e}) that 
                    $$(n,m)^k= \left \{ \begin{array}{cc} (0,km)  & \;\;\;\text{if $k $ is even} \\ (n,km)& \;\;\; \text{if $k$ is odd} .                   \end{array} \right.$$
Let $a \in \q $ and $b \in \r$. Since $m$ is odd, by (\ref{arterisco}), no $(n,m) \in R$ fixes $\l(a,b)$.  
\\
 Now, if $a=\infty$ and $b \in \r$, by (\ref{action-lines2}) we have 
 \begin{align*}
 \l(\infty,b) \in \a^R 
      & \;\;\text{iff} \;\; b=\tfrac{n}{2}.
  \end{align*}
  Therefore
      $$  \a^R =\{\l(\infty,\tfrac{n}{2})\},$$
  the space with a single point in $\r_{\infty}$.
\end{itemize}
\end{proof}



\section{Building a second model for $\underline{\underline{E}}(\zz)$} \label{chapter-build}

 In this Section we build a  model for $\underline{\underline{E}}(\zz)$ following   \cite[Sec. 2.1]{luck-weiermann}.  We recall
  that subgroups in $\vcy$ are of the form  (\ref{power-e}) or  (\ref{power-o-e}), and $\fin=\{1\}$.
 In $\fo=\vcy - \{1\} $, we  have the following equivalences:

(i) Define $H:=\langle(1,0)\rangle$. Let      $n\in \z- \{0\}$. By  (\ref{power-e}),   $$\langle(n,0)\rangle \subseteq  \langle(1,0)\rangle ,$$
   hence by (\ref{rel-eq})
       \begin{align*}
        \langle(n,0)\rangle \sim  \langle(1,0)\rangle.
        \end{align*}
   In fact these subgroups of $\fo$  are the unique subgroups that intersect $H$ in an infinite set. Denote this class by
       \begin{align} \label{h}
       [H]=[ \langle(1,0)\rangle]=[\langle(n,0)\rangle],
       \;\;\;\: \forall n \in \z- \{0\}.
       \end{align}

(ii) Fix  $n,m \in \z- \{0\}$  and
    \begin{align} \label{r}
    R=  \langle(n,2m)\rangle.
    \end{align}
Then   there is a maximal subgroup in $\fo$ which contains $R$.\\
Let  $s:=gcd(m,n)$, we define a subgroup $R'$ as follows,
\begin{align*}
R'=\langle (\frac{n}{s}, \frac{2m}{s})\rangle.
\end{align*}
By (\ref{power-e}), it is easy to see that $R'$ is the maximal subgroup in $\fo$ which contains $R$, then by (\ref{rel-eq}),
        \begin{align*}
        R  \sim \; R'.
         \end{align*}
  Observe by  (\ref{power-e}) that the only subgroups of $\fo$ related to $R'$ are the subgroups of $\fo$  which are contained in  $R'$.
  Denote by $[R]$ this class in $[\fo].$

(iii)   We have the following inclusions by  (\ref{power-e}) and  (\ref{power-o-e})
     \begin{align*}
       \langle (0,2k) \rangle& \subseteq \langle ( 0,2)\rangle,   \;\;\;\; \text{for every }\; k \in \z- \{0\};
      \\
        \langle ( 0,2(2s+1))\rangle & \subseteq \langle (r,2s+1 )\rangle,  \;\;\;\;  \text{for every } \;r ,s\in \z ;\text{  and}
      \\
        \langle ( t,2u+1)\rangle &\subseteq \langle (t,1 )\rangle \;\;\;\;  \text{for every  }\; t,u \in \z.
     \end{align*}
 By the previous inclusions and (\ref{rel-eq}) we have the following relations 
       \begin{align*}
       & \langle(0,2(2m+1))\rangle \sim   \langle(0,2)\rangle  \sim   \langle(r,1)\rangle  \text{,  and }\\
        &\langle(n,2m+1)\rangle \sim   \langle(n,1)\rangle  \sim  \langle(0,2)\rangle \sim  \langle(r,1)\rangle    \sim  \langle(r,2s+1)\rangle
         \end{align*}
 for every $n,m,r,s \in \z$. 
  Denote this class by
  \begin{align}\label{k}
    [K]=[\langle (n,2m+1)\rangle].
    \end{align}\\

We have in $[\fo]$ the classes $[H]$, $[K]$ and infinitely many countable  classes of type  $[R]$, as many as maximal subgroups in $\fo$ of the form $\langle (n,2m)\rangle$ there are, with $n,m\neq 0$ relatively prime.

 Also  $\zz$ acts on $[\fo]$ by conjugation. The classes $[H]$ and $[K]$ are fixed by conjugation and the classes of type $[R]$ are permuted.


\subsection{Explicit models } \label{models}

We  describe models for  $E(\zz)$, $E_{\g[\star]}(N_{\zz} [\star])$ and  $E(N_{\zz}[\star])$,  with $N_G[\star] $, $\mathcal{G}[\star]$  defined in (\ref{nh}),(\ref{gh})  and  $[\star] \in [\fo]$.
Let  $g=(t_1,t_2)\in \zz$.

\begin{enumerate}
\item[1.]\label{freeaction-r2}
A model for $E(\zz)$ is ${\r}^2$. 
\item[2.] A model for $E(\z \times \z)$ is $\r^2$, since $\z \times \z $ acts freely on  $\r^2$ by translation.

\item[3.]
(a) Let $[H] \in [\fo] $, where $H$ is as in (\ref{h}),  by (\ref{conj}) note: if $$(t_1,t_2)(1,0)(t_1,t_2)^{-1}=((-1)^{t_2} , 0 ),$$
then $g^{-1}Hg=H$, and therefore
\begin{align*}
N_{\zz}[H] &= \{ g\in \zz \mid |g^{-1} H g \cap H|= \infty \} \nonumber \\
&\cong  \zz
\end{align*}
So by (\ref{gh}):
$
\gh  = \vcy (H). $
\\
(b) We claim that a model for  $E_{\vcy(H)}(\zz)$ is $\r$.  Define the action of $\zz$ on $\r$ as follows
$$(t_1,t_2) x = t_2 +x,  \;\;\;\; (t_1,t_2)\in \zz, \;\; x\in \r. $$
Observe that  any point  $x\in \r$ is fixed by $(t_1,t_2)$ if only if $t_2=0$.  If $S$ is a subgroup  not in $\vcy(H)$, then the $S$-fixed point set $R^S $ is the empty set. \\
Let $S\in \vcy(H)$, then  $S=\langle (n,0)\rangle$ for some $n \in \z -\{0\}$. Then the  $S$-fixed point set  $R^S = \r$ is contractible. Therefore, by Proposition \ref{luck-caract},
 we have the claim.

\item[4.]
(a) Let $R=\langle(n,2m) \rangle$, with $n,m $ fixed in $\z- \{0\}$,   and suppose $R$ is maximal in $\fo$.
By (\ref{conj}) we have that
\begin{align*}
  (t_1,t_2)(n,2m)(t_1,t_2)^{-1}= ((-1)^{t_2}n, 2m ).
\end{align*}
Then we have two possibilities,
\begin{align*}
 \text{ i) } &gRg^{-1}=R \; \;\text{if only if } \:   t_2  \text{ is even },  \\
 \text{ ii) } &gRg^{-1}=\langle (-n,2m) \rangle \;\;  \text{if only if  } \: t_2 \; \text{is odd},
  \\
 &\text{therefore}\;\;\; gRg^{-1}\cap R=\{1\}.
\end{align*}

We conclude by (\ref{nh}) that
\begin{align*}
  N_{\zz}[R]=\{(t_1,2t_2) \mid t_1,t_2 \in \z\} = \z \times \z,
  \end{align*}
and so by (\ref{gh}) and (\ref{power-e}): 
 \begin{align*}
  \mathcal{G}[R]&= \{ \langle(ln,2lm) \rangle  \mid \; l \in \z  \} \nonumber \\
  &= \vcy(R)  \label{gr}
\end{align*}
(b) A model for $E_{\vcy(R)}(\z\times \z)$ is $\r$ by Proposition \ref{luck-caract}.  Observe  that the normalizer $N_{\zz}(R)$
of $R$ is equal to $N_{\zz}[R]$, therefore
 $R$ is a normal subgroup of $N_{\zz}[R]=\z\times \z$ and   we have the following exact sequence
\begin{align*}
0\rightarrow R \xrightarrow{i}  \z \times \z \xrightarrow{\phi}  \z \rightarrow 0
\end{align*}
where $i$ is the inclusion and $\phi$ is the projection onto the quotient $(\z \times \z) / R   \cong \z .$
Let $(t_1,t_2) \in \z \times \z$ and $x \in \r$, define the action of $\z \times \z$ in $\r$ as follows:
\begin{align*}
  (t_1,t_2) \cdot x= \phi(t_1,t_2)+ x.
\end{align*}
So, $  \phi(t_1,t_2) +x =x \;\;\;\text{if only if} \;\;\; (t_1,t_2)\in \ker \phi=R,$
therefore  the isotropy groups are in $\vcy(R)$.
Furthermore, $\r^{R}=\r$ is contractible, and if $S\not\in \vcy(R)$ then $\r^S=\emptyset $.

\item[5.] (a) Let $[K]$ be as in  (\ref{k}) and  $(t_1,t_2)\in \zz$. By equation (\ref{conj}), it  follows that
$$(t_1,t_2)(0,2)(t_1,t_2)^{-1}=(0,2). $$
Since $[\langle(0,2)\rangle]=[K]$, we conclude by (\ref{nh}) that
\begin{equation*} \label{nk}
N_{\zz}[K]=\zz,
\end{equation*}
so by (\ref{nh}), (\ref{power-e}) and (\ref{power-o-e}), note: 
\begin{align*}
  \mathcal{G}[K]&= \{D \subseteq \zz \mid D\in \fo, \: |D\cap K|=\infty \} \cup \{1\}  \nonumber \\
  &= \{ \langle (n,2m+1)\rangle \mid n,m \in \z\} \cup\{1\}. \label{gk}
\end{align*}
\item[(b)]  A model for $E_{\g[K]}\zz$ is as follows: \\
 By (\ref{power-o-e}), note that  $ \langle (n,2m+1)\rangle \subseteq \langle (n,1)\rangle$ for all $n,m \in \z$.
 Let  $K_n=\langle (n,1)\rangle $ and
  consider a point $k_n$ for each $K_n$.  Observe  for $g=(t_1,t_2)\in \zz$,
  $gK_ng^{-1}=K_{m}$  that $m=(-1)^{t_2}n+2t_1$ by Remark \ref{remconj}. Then $\zz$ acts by permuting the subgroups
  $K_n$, $n\in \z$.
\\
  We  define an action of $\zz$ on $X:=\{k_n \mid n\in \z\}$ as follows
    $$g \cdot k_n =k_{m} \, \text{ iff } \, gK_ng^{-1}=K_{m} .$$
 \\
By the   above observation and (\ref{power-o-e}), note: 
  $$g.k_n=k_n \text{ iff }  g\in N_{\zz}(K_n)=K_n.$$
  Therefore   the $\zz$-set $X$ is a model for $I_{\g[K]-\{1\}}$.
  Since a model for $E(\zz)$ is $\r^2$, by Theorem \ref{join-g-f} we conclude that a model for
  $E_{\g[K]}(\zz)$ is the join $X * \r^2$.
\end{enumerate}

\subsection{A second model for $\underline{\underline{E}}(\zz)$ }

 We are now ready to apply Theorem \ref{thm-luck-weierm} and obtain a model for $\underline{\underline{E}}(\zz)$ by the following $(\zz)$-pushout.
Let $G=\zz$ and $A=\z\times \z$, then
\begin{align}
       \xymatrix{
                 G \times_{G} EG \;
                 \coprod \; G\times_{G} EG \;
                 \coprod_{l\in I}\; G\times_{A} EA
                 \ar[r]^(0.7){i}
                 \ar[d]|{id \times_{G} p
                 \coprod   id \times_{G} g  677
                 \coprod_{k \in I} id \times_{A}    f_l   }
                 & \r^2
                 \ar[d]
                 \\
                 G \times_G E_{\vcy(H)}G \;
                 \coprod \;  G\times_G E_{\g[K]}G \;
                 \coprod_{l\in I} \; G\times_A E_{\vcy(R)}A
                 \ar[r]
                 & \underline{\underline{E}}(\zz),
}
\end{align}
where $I$ is a complete system of representatives $[R]$ of the $G$-orbits under conjugation over the classes  of subgroups of type $[R]=[ \langle (n,2m)\rangle ]$, $(n,m) \in \zz- \{(0,0)\}.$

Applying  models given in   Section  \ref{models}  and the fact
$G \times_G Y = Y$,
we have 
\begin{align}\label{pushout}
\xymatrix{
                  \r^2 \;
                 \coprod \;  \r^2 \;
                 \coprod_{l\in I}\; G\times_{A} \r^2
                 \ar[d]|{    p
                 \coprod    g
                 \coprod_{k \in I} id \times_{A}    f_l  }
                  \ar[r]^(0.7){i}
                 & \r^2
                 \ar[d]
                 \\
                  \r \;
                 \coprod \;   (\{k_n\}_{n\in \z} \ast \r^2 )\;
                 \coprod_{l \in I} \; G\times_A \r
                 \ar[r]
                 & \underline{\underline{E}}(\zz)
}
\end{align}

The maps are  given by the universal property of classifying spaces, applied with    inclusions of families of subgroups, these are as follows:

\begin{enumerate}
\item[1.]  $p\colon \r^2 \to \r$ is the projection on the $y$-axis, $(t,s)\mapsto s $.  Since $G$ acts on the $y$-axis of $\r^2$ by translation   and the action of $G$ on $\r$ is also by translation,  then $p$ is a  $G$-map, which is cellular.  By Remark \ref{incfam} $p$ is unique up to $G$-homotopy.
\item[2.]  The map $g \colon \r^2 \to \{k_n\}_{n\in \z} \ast \r^2$ is the inclusion, because the $G$-action is the same on $\r^2$, we have that $g$ is a $G$-map.
\item[3.] Let $l \in I$,     $f_l \colon  \r^2 \to \r $ are the quotient map of $\r^2$ on the line through $(n,2m)$ and the origin. It follows the map is $G$-equivariant.
\item[4.] The map $i$ is the identity on the first two $G$-spaces corresponding to the disjoint union, and  it is the natural $G$-map  on the third $G$-space. 
\end{enumerate}

\begin{thm}\label{second-model}
Let $G=\zz$, $A=\z\times \z$ and $I$ be a complete system of representatives $[R]$ of the $G$-orbits under conjugation over the classes  of subgroups of type $[R]=[ \langle (n,2m)\rangle ]$, $(n,m) \in \zz- \{(0,0)\}.$ From the  $G$-pushout (\ref{pushout}) we have
\begin{align*}
\underline{\underline{E}} (\zz)= \frac{\r \coprod \{k_n\}_{n\in \z}\ast \r^2 \coprod_{l\in I} G\times_A \r^2 }{\forall x\in \r^2: p(x)\sim g(x)\sim [1_G,f_l(x)]},
\end{align*} 
where the maps $p$, $g$ and $f_l$ are as before. 
\end{thm}

\section{Homology}
In this Section we compute  the homology groups of the model given in Theorem \ref{thm-join}, 
 $\underline{\underline{B}} (\zz)=(\coprod_{J} \s^1 )* \mathcal{K}$,  where $J$ is an
 countably infinite set. To simplify notation, denote  $X= \coprod_{J}\s^1$.

 Since the Klein bottle is path-connected, then the join $X*\mathcal{K}$ is simply connected, \cite[Sec. 7.2]{homotopical}. Therefore $H_0(X*\mathcal{K})= \z  $ and $ H_1(X*\mathcal{K})=0.$ 

From the well known short exact sequence of the join,  see \cite[Ch. 8]{munkres}, we have:  let $n>0$,
\begin{align*}
\xymatrix{0  \ar[r] &\widetilde{H}_{n+1} (X * \mathcal{K}  )  \ar[r]& \widetilde{H}_n(X\times \mathcal{K})
\ar[r]^(0.4){\pi}& \widetilde{H}_n (X) \oplus \widetilde{H}_n (\mathcal{K})  \ar[r]& 0
}
\end{align*}
where the homomorphism $\pi $ is given by $a \mapsto (\pi_X (a), -\pi_\mathcal{K}(a))$. (Here $\pi_X$
and $\pi_\mathcal{K}$ are the homomorphisms induced in homology by the projections of $X\times \mathcal{K} $
on $X$ and $\mathcal{K}$ respectively.)
Then we have the following  exact sequence:
\begin{align}\label{homology}
0 \rightarrow \widetilde{H}_{n+1} (X * \mathcal{K})  \rightarrow \bigoplus_J \widetilde{H}_n(\s^1\times \mathcal{K})
\rightarrow \left(\bigoplus_J \widetilde{H}_n (\s^1)\right) \oplus \widetilde{H}_n (\mathcal{K})
\rightarrow 0.
 \end{align}
Thus,  by  the CW structure of $\mathcal{K} $ and $\s^1$, we  use the K\"unneth theorem to obtain the  homology groups of the product $\s^1 \times \mathcal{K}$: 
\begin{align*}
 H_i(\s^1 \times \mathcal{K})=\left \{ \begin{array}{cc} \z, \;\;\;& i=0;
 \\ \z \oplus \z\oplus \z_2,\;\; &i=1;
 \\ \z\oplus \z_2, \; & i=2;
 \\ 0,   & i>2.
 \end{array}\right.
 \end{align*}
Then for  $n=1$:
\begin{align*}
0 \rightarrow \widetilde{H}_{n+1} (X * \mathcal{K}  ) \xrightarrow{} \bigoplus_J (\z \oplus \z\oplus \z_2 )
\xrightarrow{\pi}  (\bigoplus_J \z  ) \oplus \z\oplus \z_2 
\rightarrow 0,
\end{align*}
and  we  conclude that
 \begin{align*}
 H_2(X * \mathcal{K})= \bigoplus_{J'} (\z\oplus \z_2),
 \end{align*}
where $J'$ is a countably infinite set.   If  $n=2$ in  (\ref{homology}), then
$$H_3(X*\mathcal{K})= H_2(X\times \mathcal{K})=\bigoplus_J(\z\oplus \z_2).$$ And at last, for $i>3$,
$H_i(X*\mathcal{K})=0$.

\begin{prop}\label{prop-homology}
Homology groups of  $\underline{\underline{B}} (\zz)=\big(\coprod_{J} \s^1 \big)* \mathcal{K},$   where $J$ is an
 countably infinite set (Thm. \ref{thm-join}), are the following: 
 
 \begin{align*}
 H_i\big(\big(\coprod_{J} \s^1 \big)\ast \mathcal{K}\big)=\left \{ \begin{array}{ccc} \z , \;\;\;& i=0;\\
 0, \;\;\;& i=1; \\
 \bigoplus_{J'} (\z\oplus \z_2), \;\;\;& i=2;
  \\
 \bigoplus_{J} (\z\oplus \z_2), \;\;\;& i=3;
  \\ 0,   & i>3,
 \end{array}\right.  
 \end{align*}
where $J'$ is a countable infinite set. 
\end{prop}

\end{document}